\documentclass[a4paper,11pt]{article}

\usepackage[utf8]{inputenc}
\usepackage[english]{babel}
\usepackage{fullpage}
\usepackage{xspace}
\usepackage{amsmath}
\usepackage{amsthm}
\usepackage{amssymb}
\usepackage[bookmarks=true,
            bookmarksnumbered=true,
            pagebackref=true,
            colorlinks=true,
            linkcolor=blue,
            citecolor=red]{hyperref}

\title{``Weak yet strong'' restrictions of Hindman's Finite Sums Theorem}

\author{Lorenzo Carlucci \\
  Dipartimento di Informatica, Sapienza --- Universit\`a di Roma.\\
  \texttt{carlucci@di.uniroma1.it}}

\date{\today{}}

\hypersetup{
  pdfauthor={Lorenzo Carlucci}
  pdftitle ={}
}

\newtheorem{definition}{Definition}
\newtheorem{proposition}{Proposition}
\newtheorem{theorem}{Theorem}

\newtheorem{claim}{Claim}
\newtheorem{open.problem}{Open Problem}

\newtheorem*{theorem*}{Theorem}
\newtheorem*{corollary*}{Corollary}
\newtheorem*{proposition*}{Proposition*}
\newtheorem*{lemma*}{Lemma}
\newtheorem*{fact*}{Fact}
\newtheorem*{claim*}{Claim}
\newtheorem*{open.problem*}{Open Problem}
\newtheorem*{remark*}{Remark}
\newtheorem*{example*}{Example}
\newtheorem*{exercise*}{Exercise}

\newcommand\Nat{\mathbf{N}}
\newcommand\RCA{\mathsf{RCA}}
\newcommand\ACA{\mathsf{ACA}}

\newcommand\HT{\mathsf{HT}}

\newcommand\RT{\mathsf{RT}}

\begin{document} 


\maketitle

\begin{abstract}
We present a natural restriction of Hindman's Finite Sums Theorem 
that admits a simple combinatorial proof (one that does not also prove 
the full Finite Sums Theorem) and low computability-theoretic and proof-theoretic
upper bounds, yet implies the existence of the Turing Jump, thus realizing the
only known lower bound for the full Finite Sums Theorem. This is the first example of
this kind. In fact we isolate a rich family of similar restrictions of Hindman's 
Theorem with analogous properties. 
\end{abstract}



\section{Introduction and Motivation}
The following question was asked by Hindman, Leader and Strauss in \cite{Hin-Lea-Str:03}:

\begin{quote}
{\em Question 12}. Is there a proof that whenever $\Nat$ is finitely coloured there is a sequence
$x_1, x_2, \dots$ such that all $x_i$ and all $x_i + x_j$ ($i\neq j$) 
have the same colour, that does not also prove the Finite Sums Theorem?
\end{quote}
The theorem referred to as the Finite Sums Theorem is the famous result of Hindman's
(the original proof is in \cite{Hin:74}) stating that whenever $\Nat$ is finitely coloured there is a sequence
$x_1, x_2, \dots$ such that all finite non-empty sums of distinct elements from the sequence
have the same colour. We will sometimes refer to this statement as Hindman's Theorem, or 
the full Hindman's Theorem.

In this paper we present some results that are related to Question 12 above. We isolate a rich 
family $\mathcal{F}$ of natural restrictions of the Finite Sums Theorem with the following two properties:

\begin{enumerate}
\item Each member of the family $\mathcal{F}$ admits a simple combinatorial proof that does not establish
Hindman's Theorem, but
\item Each member of a non-trivial sub-family of $\mathcal{F}$ 
is strong in the sense of having the same computability-theoretic lower bounds 
that are known to hold for Hindman's Theorem.
\end{enumerate}

The simplicity of the proof referred to in point (1) above is evident in the sense that all members of 
$\mathcal{F}$ admit a proof consisting in a finite iteration of the Infinite Ramsey's Theorem 
and an application of some classical theorem from Finite Combinatorics. Yet, much more detailed information
can be obtained by using the tools of Computability Theory and Reverse Mathematics, the areas
where the lower bound mentioned in point (2) above come from. 

The strength of the Finite Sums Theorem is indeed a major open problem in these areas (see \cite{Mon:11:open}, Question 9).
A huge gap remains between the known lower and upper bounds on the computability-theoretic 
and proof-theoretic strength of Hindman's Theorem \cite{Hin:74}. 
Blass, Hirst and Simpson in~\cite{Bla-Hir-Sim:87} established the 
following lower and upper bounds thirty years ago:

\begin{enumerate}
\item There exists a computable coloring $c:\Nat\to 2$ such that any solution to 
Hindman's Theorem for $c$ computes $\emptyset'$, the first Turing Jump of the computable sets. 

\item For every computable coloring $c:\Nat\to 2$ there exists a solution set
computable from $\emptyset^{(\omega)}$, the $\omega$-th Turing Jump of the computable sets.
\end{enumerate}

By a ``solution to Hindman's Theorem for coloring $c$'' we mean an infinite set $H$ 
such that all finite non-empty sums of elements from $H$ have the same $c$-color.
As often is the case, the above computability-theoretic results have direct 
corollaries in Reverse Mathematics (see \cite{Sim:SOSOA,Hir:STT:14} for excellent introductions
to the topic). 
Letting $\HT$ denote the natural formalization of Hindman's Finite Sums Theorem
in the language of arithmetic, the only known upper and lower bounds on the logical 
strength of the full Finite Sums Theorem are the following (again from \cite{Bla-Hir-Sim:87}):
$$ \ACA_0^+ \geq \HT \geq \ACA_0.$$
Recall that $\ACA_0$ is equivalent to $\RCA_0+\forall X \exists Y (Y=X')$ and that
$\ACA_0^+$ is equivalent to $\RCA_0 + \forall X \exists Y (Y = X^{(\omega)})$. 
Note that restricting the consideration to $2$-colorings is inessential, since
$\HT_2$ already implies $\ACA_0$.

Recently there has been some interest in the computability-theoretic and proof-theoretic
strength of restrictions of Hindman's Theorem
(see~\cite{Hir:12:HvsH,DJSW:16,Car:16:weak}). While~\cite{Hir:12:HvsH} deals with
a restriction on the sequence of finite sets in the Finite Unions formulation of Hindman's 
Theorem, both~\cite{DJSW:16} and~\cite{Car:16:weak} deal with restrictions on the types of
sums that are guaranteed to be colored the same color. 

Blass conjectured in~\cite{Bla:05} that the complexity of Hindman's Theorem might 
grow with the length of the sums for which homogeneity is guaranteed. Let us denote
by $\HT^{\leq n}_r$ the restriction of the Finite Sums Theorem to colorings with 
$r$ colors and sums of at most $n$ terms. The conjecture discussed in \cite{Bla:05} is then
that the complexity of $\HT^{\leq n}_r$ is growing with $n$.  

The main result in~\cite{DJSW:16} is that the above described $\emptyset'$ lower bound known 
to hold for the full Hindman's Theorem already applies to its restriction to $4$ colors and to 
sums of at most $3$ terms ($\HT^{\leq 3}_4$ in the notation introduced above). 
On the other hand, note that no upper bound other than 
the upper bound for full Hindman's Theorem is known to hold for this restricted version, 
and the same is true for $\HT^{\leq 2}_2$, the restriction to sums of at most $2$ terms! 
This is obviously related to Question 12 of~\cite{Hin-Lea-Str:03} quoted above. 

On the other hand, the variants studied by Hirst in~\cite{Hir:12:HvsH} (called
Hilbert's Theorem) and by the author 
in~\cite{Car:16:weak} (called the Adjacent Hindman's Theorem) do admit simple proofs, 
but are very weak and provably fall short 
of hitting the known lower bounds for the full Hindman's Theorem 
(they are provable, respectively, from the Infinite Pigeonhole
Principle and from Ramsey's Theorem for pairs).

By contrast, the family of natural restriction of the Finite Sums Theorems
introduced in the present paper has members that are ``weak'' in the sense of admitting easy proofs 
yet ``strong'' with respect to computability-theoretic and proof-theoretic lower bounds. 

In terms of Computability Theory and Reverse Mathematics, the properties of our
family of restrictions of the Finite Sums Theorem are summarized as follows: 
{\em All} members of the family have upper bounds in the Arithmetical Hierarchy 
for computable instances and proofs in $\ACA_0$. Yet {\em many} members of such family imply 
the existence of the Turing Jump. In terms of Reverse Mathematics, they imply $\ACA_0$.

\section{A family of restrictions of the Finite Sums Theorem}

The present section is organized as follows. We first formulate, in section \ref{sub:hinbrau},
a particular restriction of 
the Finite Sums Theorem, called the Hindman-Brauer Theorem, and prove it by a simple finite
iteration of the Infinite Ramsey's Theorem plus finitary combinatorial tools (Theorem \ref{thm:hinbrau}). 
The argument is indeed general and in section \ref{sub:family} we describe the family of statements that can be proved
by exactly the same proof. This proof is obviously much simpler than any known proof of the Finite 
Sums Theorem and does not establish the latter. 
In the last subsection (Section \ref{sub:uppbds}) 
we extract computability-theoretic and proof-theoretic upper bounds for each member of the family.

\subsection{An example: the Hindman-Brauer Theorem}\label{sub:hinbrau}
We start with a particular example. We will use the following theorem, due to Alfred Brauer~\cite{Bra:28}, 
which is a joint strengthening of Van Der Waerden's~\cite{VdW:27}
and Schur's~\cite{Sch:16} theorems.

\begin{theorem}[Brauer's Theorem, \cite{Bra:28}]
For all $r,\ell,s\geq 1$ there exists $n=n(r,\ell,s)$ such that
if $g:[1,n]\to r$ then there exists $a,b>0$ such that
$\{a,a+b,a+2b,\dots,a+ (\ell-1)b\}\cup \{sb\}\subseteq [1,n]$ is monochromatic.
\end{theorem} 

Let $B:\Nat^3\to \Nat$ denote the witnessing function for Brauer's Theorem. For $n=2^{t_1}+\dots + 2^{t_k}$ with $t_1 < \dots < t_k$
let $\lambda(n)=t_1$ and $\mu(n)=t_k$. The following Apartness Condition is crucial in what follows. 

\begin{definition}[Apartness Condition]
We say that a set $X$ satisfies the Apartness Condition (or {\em is apart}) if for all $x,x'\in X$, 
if $x<x'$ then $\mu(x)<\lambda(x')$. 
\end{definition}
Note that the Apartness Condition is inherited by subsets.
If $a$ is a positive integer and $B$ is a set we denote by $FS^{=a}(B)$ (resp. $FS^{\leq a}(B)$)
the set of sums of exactly (resp.~at most) $a$ distinct elements from $B$. 
More generally, if $A$ and $B$ are sets we denote by $FS^{ A}(B)$ the set of all sums of $j$-many distinct terms 
from $B$, for all $j\in A$. Thus, e.g., $FS^{ \{1,2,3\}}(B)$ is another name for $FS^{\leq 3}(B)$.
By $FS(B)$ we denote $FS^\Nat(B)$, the set of all non-empty finite sums of distinct elements of $B$.
By $\RT^n_r$ we denote the Infinite Ramsey's Theorem for $r$-colorings of $n$-tuples.

\begin{theorem}[Hindman-Brauer Theorem]\label{thm:hinbrau}
For all $c:\Nat \to 2$ there exists and infinite and apart set
$H\subseteq\Nat$ such that for some $a,b>0$ the set
$FS^{\{a,a+b,a+2b\}\cup\{b\}}(H)$ is monochromatic.
\end{theorem}

\begin{proof}
Let $c:\Nat \to 2$ be given. Let $k=B(2,3,1)$. Consider the following construction.

Let $H_0$ be an infinite (computable) set satisfying the Apartness
Condition, e.g. $\{2^t \,:\, t\in\Nat\}$.

Let $H_1\subseteq H_0$ be an infinite homogeneous set for $c$, witnessing $\RT^1_2$ relative 
to $H_0$

Let $f_2 : [\Nat]^2\to 2$ be defined as $f(x,y) = c(x+y)$. Let $H_2\subseteq H_1$
be an infinite homogeneous set for $f_2$, witnessing $\RT^2_2$ relative to $H_1$. 

Let $f_3 : [\Nat]^2\to 2$ be defined as $f(x,y,z) = c(x+y+z)$. Let $H_3\subseteq H_2$
be an infinite homogeneous set for $f_3$, witnessing $\RT^3_2$ relative to $H_2$.

We continue in this fashion for $k$ steps. This determines a finite sequence of infinite sets
$H_0,H_1,\dots,H_k$ such that
$$ H_0 \supseteq H_1 \supseteq H_2 \supseteq \dots \supseteq H_k.$$
Each $H_i$ satisfies the Apartness Condition. 
Furthermore, for each $i\in [1,k]$ we have that $FS^{=i}(H_j)$ is monochromatic under $c$ for all $j\in [i,k]$.
Also, $FS^{=i}(H_k)$ is monochromatic for each $i\in [1,k]$. 
Let $c_i$ be the color of $FS^{=i}(H_k)$ under $c$. 

The construction can be seen as defining a coloring
$C:[1,k]\to 2$, setting $C(i)=c_i$. Since $k=B(2,3,1)$, by Brauer's Theorem there exists $a,b>0$ in $[1,k]$ such that
$\{a,a+b,a+2b\}\cup\{b\}\subseteq [1,k]$ is monochromatic for $C$. Let $i<2$ be the color. 
Then $FS^{\{a, a+b, a+2b\}\cup\{b\}}(H_k)$ 
is monochromatic of color $i$ for the original coloring $c$.
\end{proof}

A comment is in order: the above proof is obviously {\em much simpler} than any known proof establishing Hindman's Theorem. 
It is also obvious that it does not establish the full Finite Sums Theorem. 
Note that nothing in the above construction is special about $2$ colors and $3$-terms
arithmetic progressions. 

\subsection{A family of restrictions of the Finite Sums Theorem admitting simple proofs}\label{sub:family}

The proof of Theorem \ref{thm:hinbrau}
is easily adapted to arbitrary values $r$ for number of colors and
$\ell$ for the length of the arithmetic progression. More importantly
one can substitute Brauer's Theorem by {\em virtually any} theorem about finite colorings 
of numbers from the literature 
(Schur's Theorem \cite{Sch:16}, Van der Waerden's Theorem \cite{VdW:27}, Folkman's Theorem \cite{San:68,Gra-Ron-Rot-Spe:80}, 
just to name a few), 
yielding a rich family of Hindman-type theorems. 

The general form of the restrictions of the Finite Sums Theorem obtained by the proof of Theorem \ref{thm:hinbrau}
is the following:
\begin{quote}
For all $c:\Nat\to r$ there exists an infinite and apart $H\subseteq\Nat$ and there exists a finite $A$, 
satisfying some specific conditions, such that $FS^{ A}(H)$ is monochromatic. 
\end{quote}
For each set $A\subseteq\Nat$ and positive integer $r>0$, we let $\HT^{ A}_r$ denote such a statement.
We describe the family by presenting a list of some of its typical members, grouped
by sub-families. The general pattern will be clear enough.

\begin{center}
{\bf Schur Family}: 
\end{center}
For each positive integer $r$ let $\HT^{\{a,b,a+b\}}_r$ denote the following statement.
\begin{quote}
Whenever $\Nat$ is colored in $r$ colors there is an infinite and apart set $X=\{x_1,x_2,\dots\}_<$
and positive integers $a,b$ such that all elements
of $FS^{\{a,b,a+b\}}(\{x_1,x_2,\dots\})$ have the same color. 
\end{quote}

\begin{center}
{\bf Van der Waerden Family}:
\end{center}
For each pair of positive integers $r,\ell$ let $\HT^{\{a,b,a+b,\dots,a+(\ell-1)b\}}_r$ denote
the following statement.
\begin{quote}
Whenever $\Nat$ is colored in $r$ colors there is an infinite and apart set $X=\{x_1,x_2,\dots\}_<$
and positive integers $a,b$ such that all elements
of $FS^{\{a,a+b,a+2b,\dots,a+(\ell-1)b\}}(\{x_1,x_2,\dots\})$ have the same color. 
\end{quote}

\begin{center}
{\bf Brauer Family}: 
\end{center}
For each pair of positive integers $r,\ell$, let $\HT^{\{a,b,a+b,\dots,a+(\ell-1)b\}\cup\{b\}}_r$ denote
the following statement.
\begin{quote}
Whenever $\Nat$ is colored in $r$ colors there is an infinite and apart set $X=\{x_1,x_2,\dots\}_<$
and positive integers $a,b$ such that all elements
of $FS^{\{a,a+b,a+2b,\dots,a+(\ell-1)b\}\cup\{b\}}(\{x_1,x_2,\dots\})$ have the same color. 
\end{quote}

\begin{center}
{\bf Folkman Family}:
\end{center}
For each pair of positive integers $r,\ell$, let $\HT^{ FS(\{i_1,\dots,i_\ell\})}_r$ denote
the following statement.
\begin{quote}
Whenever $\Nat$ is colored in $r$ colors there is an infinite and apart set $X=\{x_1,x_2,\dots\}_<$
and positive integers $i_1,\dots,i_\ell$ such that all elements
of $FS^{ FS(\{i_1,\dots,i_\ell\})}(\{x_1,x_2,\dots\})$ have the same color. 
\end{quote}

\subsection{Computability-theoretic and proof-theoretic upper bounds}

The observable simplicity of the proof of Theorem \ref{thm:hinbrau} can be measured by extracting 
from it computability-theoretic and proof-theoretic upper bounds. From the finite iteration argument
given above one can glean upper bounds that are better than the known upper bounds for the 
full Finite Sums Theorem. 

To assess the Computability and Reverse Mathematics corollaries, it may be
convenient to reformulate the general argument of Theorem \ref{thm:hinbrau} as follows
(again, we only give the details for the case of the Hindman-Brauer Theorem):

\begin{proof}[Second proof of Theorem \ref{thm:hinbrau}]
Let $n$ be a positive integer. Given $c:\Nat\to 2$ let $g_n:[\Nat]^n\to 2^n$ be defined as follows: 
$$ g_n(x_1,\dots,x_n) = \langle c(x_1),c(x_1+x_2),\dots,c(x_1+\dots+x_n)\rangle.$$
Fix an infinite and apart set $H_0$ of positive integers. 
By $\RT^n_{2^n}$ relativized to $H_0$ we get an infinite apart 
set $H$ monochromatic for $g_n$.
Let the color be $\sigma=(c_1,\dots,c_n)$, a binary sequence of length $n$. 
Then, for each $i\in [1,n]$, $g_n$ restricted to $FS^{=i}(H)$ is monochromatic 
of color $c_i$. The sequence $\sigma$ is a coloring of $n$ in $2$ colors. 
If $n=B(2,3,1)$ then, by the finite Brauer's Theorem, 
there exists $a,b>0$ in such that $\{a,a+b,a+2b\}\cup \{b\}\subseteq [1,n]$ and 
$$c_a = c_b = c_{a+b} = c_{a+2b}.$$
Then $FS^{ \{a,a+b,a+2b,b\}}(H)$ is monochromatic of color $c_a$.
\end{proof}

The above argument 
shows that $\RT^{B(2,3,1)}_{2^{B(2,3,1)}}$ implies the Hindman-Brauer Theorem $\HT^{\{a,a+b,a+2b\}\cup\{b\}}_2$. 
The difference from the previously given argument is that we have only used
one instance of Ramsey's Theorem, albeit for a larger number of colours. 

We can then quote the following classical results of Jockusch's about upper bounds 
on the computability-theoretic content of Ramsey's Theorem (see \cite{Joc:72}). 

\begin{theorem}[Jockusch, \cite{Joc:72}]\label{joc1}
Every computable $f:[\Nat]^n\to r$ has an infinite $\Pi^0_n$ homogeneous set.
\end{theorem}

\begin{theorem}[Jockusch, \cite{Joc:72}]\label{joc2}
Every computable $f:[\Nat]^n\to r$ has an infinite homogeneous set $H$ such that $H'\leq_T \emptyset^{(n)}$.
\end{theorem}

Then we have the following proposition as an immediate corollary, where $\leq_T$ denotes 
Turing reducibility.

\begin{proposition}
Every computable $c:\Nat\to 2$ has an infinite and apart set $H$ such that 
for some $a,b>0$ the set $FS^{\{a,b,a+b,a+2b\}}(H)$ is monochromatic
and such that $H' \leq_T \emptyset^{(B(2,3,1))}$.
\end{proposition}

Analogously we get arithmetical upper bounds for the other theorems admitting a 
similar proof, some of which are apparently quite strong (e.g., the one derived 
from Folkman's Theorem described above). This should be contrasted with the fact 
that there are no similar upper bounds on the computability-theoretic content of 
Hindman's Theorem, not even when restricted to sums of at most two terms! 
Again, for the latter two theorems, the only upper bound for general computable
solutions is $\emptyset^{(\omega)}$.

We now comment on Reverse Mathematics implications. 
The argument described above is formalizable in $\ACA_0$
(note that most of the finite combinatorial theorems quoted above are provable in 
$\RCA_0$). We then get, for every standard $k\in\Nat$, that
$$\RCA_0 \vdash \RT^{B(2,k,1)}_{2^k} \to \HT^{\{a,a+b,\dots,a+(k-1)b\}\cup\{b\}}_2.$$
Thus we have the following proposition. 
\begin{proposition}
For each standard $k\in\Nat$:
$$\ACA_0 \vdash \HT^{\{a,a+b,a+2b,\dots,a+(k-1)b\}\cup\{b\}}_2.$$
\end{proposition}
Again, this should be contrasted with the $\ACA_0^{(\omega)}$ upper bound that
is known to hold for the full Finite Sums Theorem, as well as for its restriction
to sums of at most two terms. 

Obviously, similar proof-theoretic upper bounds hold for
many other members of the family by the same argument, as long as the 
underlying finite combinatorial principle does not itself require strong
axioms.

\section{A lower bound on the Hindman-Brauer Theorem}

Let $K$ denote the (computably enumerable but not computable) Halting Set or, equivalently, 
the first Turing jump $\emptyset'$.
We show that there exists a computable coloring $c:\Nat \to 2$ such that $K$ is computable 
from any solution $H$ of the Hindman-Brauer Theorem $\HT^{\{a,a+b,a+2b,b\}}_2$ for the instance $c$, 
i.e., $H$ is infinite and apart and for some $a,b>0$, 
the set $FS^{\{a,a+b,a+2b,b\}}(H)$ is monochromatic. 

We adapt the beautiful proof of the lower bound for the full Hindman's Theorem
by Blass, Hirst and Simpson (Theorem 2.2 in \ref{Bla-Hir-Sim:87}). 
Gaps and short gaps of numbers are defined as in \cite{Bla-Hir-Sim:87}. We recall the 
definitions for convenience. Fix an enumeration of the computably enumerable set $K$
and denote by $K[k]$ the set enumerated in 
$k$ steps of computation by this algorithm. If $n=2^{t_1}+\dots + 2^{t_k}$ with $t_1 < \dots < t_k$
we refer to pairs $(t_i,t_{i+1})$ as the {\em gaps} of $n$. A gap $(a,b)$ of $n$ is 
{\em short in} $n$ if there exists $x\leq a$ such that $x\in K$ but
$x\notin K[b]$. A gap $(a,b)$ of $n$ is {\em very short in} $n$ if
there exists $x\leq a$ such that $x\in K[\mu(n)]$ but
$x\notin K[b]$. A gap of $n$ that is short in $n$ is called {\em a short
gap of} $n$. Let $SG(n)$ denote the set of short gaps of $n$. A gap of $n$ that is
very short in $n$ is called {\em a short gap of} $n$. Let $VSG(n)$ denote the set of 
very short gaps of $n$. Notice that given $n$ one can effectively compute $VSG(n)$
but not $SG(n)$. 

\begin{theorem}\label{thm:lowerbound}
There exists a computable coloring $c:\Nat\to 2$ such that if $H\subseteq \Nat$ 
is a solution to the Hindman-Brauer Theorem 
for instance $c$ then $K$ is computable from $H$. 
\end{theorem}

\begin{proof}
Consider the following computable coloring of $\Nat$ in $2$ colors. 
$$c(n)= VSG(n)\mod 2.$$

Let $H\subseteq \Nat$ and $a,b>0$ be such that $H$ is infinite, satisfies the Apartness Condition, and is such that all sums of size $a,b,a+b,a+2b$ of elements from $H$ have the same color
under $c$.

\begin{claim}
For every $m\in FS^{=a}(H)$, $SG(m)$ is even. 
\end{claim}

\begin{proof}
Pick $n$ in $FS^{=b}(H)$ so large that the following three points are satisfied:

\begin{enumerate}
\item $\mu(m) < \lambda(n)$, 
\item for all $x\leq \mu(m)$, $x\in K$ if and only if $x\in K[\lambda(n)]$, 
\item $\mu(m+n)=\mu(n)$. 
\end{enumerate}

This choice is legitimate since $H$ satisfies the Apartness Condition and is infinite.
Since $m\in FS^{=a}(H)$ there exists $t_1 < t_2 < \dots < t_a$ elements of $H$
such that $m= t_1+t_2+\dots + t_a$. Since $H$ satisfies the Apartness Condition, 
we have that $\mu(m) = \mu(t_a)$ and $\lambda(m)=\lambda(t_1)$. (Analogous equations
hold for sums of type $b$, $a+b$, $a+2b$). Now observe that elements of $FS^{=b}(H)$
are unbounded with respect to their $\lambda$-projection, i.e. for all $d$ there
there exists $q\in FS^{=b}(H)$ such that $\lambda(q)>d$. This follows from the fact
that $H$ satisfies the Apartness Condition and by the previous observations on $\lambda$-projections
of sums. So requirements 1 and 2 above can be met. To meet requirement 3 just observe that if 
$m=t_1+\dots+t_a$ with $t_1<\dots <t_a$ 
we can pick an $n\in FS^{=b}(H)$, say $n = t'_1+\dots+t'_b$ with $t'_1 <\dots <t'_b$, 
such that $\mu(m)=\mu(t_a) < \lambda(t'_1) = \lambda(n)$ because $H$ is apart.

We now compute the number of very short gaps of $m+n$, arguing as in \cite{Bla-Hir-Sim:87}.
We consider separately the gaps of $m$, the gaps of $n$ and the gap $(\mu(m),\lambda(n))$.

The gap $(\mu(m),\lambda(n))$ is not very short, by choice of $n$ (item (2) above). 

A gap of $n$ is very short in $m+n$ if and only if it is very short in $n$, since
$\mu(m+n)=\mu(n)$. 

A gap $(a,b)$ of $m$ is very short in $m+n$ if and only if it is short (not necessarily very short) as a gap of $m$: 
Suppose that $(a,b)$ is a gap of $m$ very short in $m+n$. By definition 
there exists $x\leq a$ such that $x\in K[\mu(m+n)]$ but $x\notin K[b]$.
Then there exists $x\leq a$ such that $x\in K$ but $x\notin K[b]$ hence
$(a,b)$ is short in $m$. For the other direction suppose $(a,b)$ short
in $m$, that is there exists $x\leq a$ such that $x\in K$ but $x\notin K[b]$.
Then by choice of $n$ ($\mu(m)<\lambda(n)$ by item (1) above and $\lambda(n)\mu(n)$)
we have that $x\leq a$ and $x\in K$ implies $x\in \mu(n)$. But $\mu(n)=\mu(m+n)$
by item (3) above. Hence $(a,b)$ is very short in $m$.

Therefore we have the following equation:
$$VSG(m+n) = SG(m) + VSG(n).$$ 
By hypothesis on $H$, 
$VSG(m+n)$ and $VSG(n)$ have the same parity, since $m+n\in FS^{=a+b}(H)$. 
\end{proof}

\begin{claim}
For every $m\in FS^{=b}(H)$, $SG(m)$ is even. 
\end{claim}

\begin{proof}
Pick $n$ in $FS^{=a}(H)$ so large that the following three points are satisfied:

\begin{enumerate}
\item $\mu(m) < \lambda(n)$, 
\item for all $x\leq \mu(m)$, $x\in K$ if and only if $x\in K[\lambda(n)]$, 
\item $\mu(m+n)=\mu(n)$. 
\end{enumerate}

This choice is legitimate since $H$ satisfies the Apartness Condition and is infinite.
Then argue as previously. We end up with
$$VSG(m+n) = SG(m) + VSG(n).$$ 
By hypothesis on $H$, 
$VSG(m+n)$ and $VSG(n)$ have the same parity, since $m+n\in FS^{=a+b}(H)$. 
\end{proof}

\begin{claim}
For every $m\in FS^{=a+b}(H)$, $SG(m)$ is even.  
\end{claim}

\begin{proof}
Pick $n$ in $FS^{=b}(H)$ so large that the following three points are satisfied:

\begin{enumerate}
\item $\mu(m) < \lambda(n)$, 
\item for all $x\leq \mu(m)$, $x\in K$ if and only if $x\in K[\lambda(n)]$, 
\item $\mu(m+n)=\mu(n)$. 
\end{enumerate}

This choice is legitimate since $H$ satisfies the Apartness Condition and is infinite.
Then argue as previously. We end up with 
$$VSG(m+n) = SG(m) + VSG(n).$$ 
By hypothesis on $H$, 
$VSG(m+n)$ and $VSG(n)$ have the same parity, since $m+n\in FS^{=a+2b}(H)$. 
\end{proof}

\begin{claim}\label{claim2}
For all $m\in FS^{=a}(H)$ and all $n\in FS^{=b}(H)$ such that $\mu(m)<\lambda(n)$
we have: 
$$\forall x \leq \mu(m)(x\in K \leftrightarrow x \in K[\lambda(n)]).$$
\end{claim}

\begin{proof}
By way of contradiction suppose that $(\mu(m),\lambda(n))$ is short. 
Then:
$$SG(m+n) = SG(m) + SG(n) + 1.$$
But $SG(m+n),SG(n),SG(m)$ are all even by the previous claims. Contradiction.  
\end{proof}

We now describe an algorithm showing that $K$ is computable from $H$. 
Given an input $x$, use the oracle to find an $m\in FS^{=a}(H)$ such that
$x \leq \mu(m)$ and an $n\in FS^{=b}(H)$ such that $m < n$ and $\mu(m) < \lambda(n)$. 

Then run the algorithm enumerating $K$ for $\lambda(n)$ steps to decide membership of $x\in K[\lambda(n)]$. 
By Claim \ref{claim2} this also decides membership in $K$. 
\end{proof}

As in \cite{Bla-Hir-Sim:87} a straightforward relativization of the above proof gives the 
following proposition.

\begin{proposition} Over $\RCA_0$,
$\HT^{\{a,a+b,a+2b\}\cup\{b\}}_2$ implies $\ACA_0$.
\end{proposition}

Note that the above proof works for any member $\HT^{ A}_2$ of our family 
such that $A$ is guaranteed to contain a set of the form $\{x,x+y,x+2y\}\cup\{y\}$
for some positive integers $x,y$. 

\section{Conclusions}

We have introduced a family of natural restrictions of Hindman's Finite Sums Theorem 
such that each member of the family admits a fairly simple proof, 
has arithmetical upper bounds for computable instances, 
yet many members of the family imply the existence of the Halting Set. 
These are the first examples with these properties. 
In fact, Hindman's Theorem restricted to sums of at most $3$ terms and $4$-colorings 
$\HT^{\leq 3}_4$ shares the 
same $\emptyset^{(\omega)}$ lower bound (by the main result of~\cite{DJSW:16}) but has no other proof 
(resp.~upper bound) apart from the proof (resp.~upper bound) known for the full Finite Sums 
Theorem.
Of all members $\HT^{ A}_2$ of our family we know how to prove that they achieve the only 
lower bounds known for the full Finite Sums Theorem {\em provided that} the set $A$
of lengths of sums for which homogeneity is guaranteed contains a $3$-terms 
arithmetic progression and its difference. This is the best to our current knowledge
but it is an interesting question to characterize the members in the family that
imply $\ACA_0$. 

Some members of our family are apparently strong when compared to 
the family of restrictions of Hindman's Theorem based on the mere number 
of terms in the sums studied in \cite{DJSW:16}. Compare, e.g., $\HT^{\{a,b, a+b,a+2b,a+3b,\dots,a+100b\}}_2$
with $\HT^{\leq 3}_2$. Yet this superficial impression might be misleading.
It is an easy observation that
$\HT^{ A}_2$ for an $A$ such that $A\supseteq\{a,2a\}$ for some $a>0$ implies
$\HT^{\leq 2}_2$. Analogous relations hold for $A\supseteq\{a,2a,3a\}$ and
$\HT^{\leq 3}_2$. These will be discussed in future work. 
Yet it doesn't seem obvious to get an implication 
from those $\HT^{ A}$s and the $\HT^{\leq n}$s. 
Many more non-trivial implications can be established and will reported
elsewhere.

\end{document}